\documentclass[a4paper]{amsart}
\usepackage{amssymb}
\usepackage[all]{xy}
\SelectTips{cm}{}
\calclayout


\theoremstyle{plain}% name and number in boldface and body in italic
\newtheorem{thm}{Theorem}[section]
\newtheorem{cor}[thm]{Corollary}
\newtheorem{lem}[thm]{Lemma}
\newtheorem{exm}{Example}
\newtheorem{prop}[thm]{Proposition}
\newtheorem{defn}[thm]{Definition}
\newtheorem{rem}[thm]{Remark}

\newcommand{\smatrix}[1]{\left(\begin{smallmatrix}#1\end{smallmatrix}\right)}

\def\Inj{\operatorname{Inj}}
\def\inj{\operatorname{inj}}
\def\mod{\operatorname{mod}}
\def\Mod{\operatorname{Mod}}

\def\Hom{\operatorname{Hom}}
\def\End{\operatorname{End}}

\def\supp{\operatorname{supp}}
\def\bdim{\operatorname{\bold{dim}}}
\def\uMod{\operatorname{\underline{Mod}}\nolimits}
\def\umod{\operatorname{\underline{mod}}\nolimits}
\def\mod{\operatorname{mod}}
\def\Mod{\operatorname{Mod}}

\def\Hom{\operatorname{Hom}}
\def\End{\operatorname{End}}
\def\proj{\operatorname{proj}}
\def\inj{\operatorname{inj}}

\def\Inj{\operatorname{Inj}}

\def\mabb#1{\mathbb{#1}}
\def\mac#1{\mathcal{#1}}

\def\db #1{D^b(\mod #1)}
\def\a{\alpha}
\def\b{\beta}

\def\rta{\rightarrow}
\def\a{\alpha}
\def\b{\beta}

\def\sg{\sigma}

\begin{document}

\title{Generically trivial derived categories}

\author{Zhe Han}
\address{Zhe Han \\School of Mathematics and Information Science\\
Henan University \\Kaifeng 475001\\ China.}
\address{Fakult\"at f\"ur Mathematik\\
Universit\"at Bielefeld \\D-33501 Bielefeld\\ Germany.}
\email{hanzhe0302@gmail.com}

\begin{abstract}
We study generic objects in triangulated categories and give a characterization of the finite dimensional algebras $A$ such that the  derived categories  $ D(\Mod A)$ are generically trivial. This characterization is an analogue of a result of Crawley-Boevey for module categories.
 As a consequence,
we show that $D(\Mod A)$ is generically trivial if and only if the category of perfect complexes $K^b(\proj A)$ is locally finite.

\end{abstract}
\maketitle
\setcounter{tocdepth}{1}
%\tableofcontents

\section{Introduction}
 For a finite
dimensional algebra $A$ over an algebraically closed field $k$, indecomposable pure injective modules which are not finitely generated are very important. There is a criterion to determine the representation type of a given algebra by
these pure injective modules.
 An artin algebra  $A$ is representation infinite if and only if
 there is an indecomposable pure injective $A$-module which is not finitely generated \cite{he97,cb92}. If it is of finite length over its endomorphism ring, then it is called generic \cite{cb92}.
 An algebra $A$ is called generically trivial if it has no generic modules.

There are similar notions in triangulated categories. Given a compactly generated triangulated category $\mac T$, an object $N$ of $\mac T$ is endofinite if the $\End N$-module $\Hom_{\mac T}(C,N)$ is of finite length for every $C\in\mac T^c$, where $\mac T^c$ is the subcategory of all compact objects in $\mac T $. Endofinite objects have nice decomposition properties \cite{kra99}.
An object $G$ in $\mac T$ is called a generic object if it is indecomposable, endofinite and not a compact object.  The triangulated category $\mac T$ is called generically trivial if it has no generic objects.

It is natural to ask how generic objects affect the behavior of $\mac T$. We consider the case $\mac T=  D(\Mod A)$, the unbounded derived category of all $A$-modules. For this triangulated category, we have $D(\Mod A)^c\cong K^b(\proj A)$, the category of perfect complexes. Is there any relation between the generic objects in $ D(\Mod A)$ and
 the infinite continuous families of  indecomposable objects in $K^b(\proj A)$?
We shall try to answer these questions. We will give a characterization of the algebras with generically trivial derived categories.

Our main result is the following:
\begin{thm}[Theorem 3.11]
  Let A be a finite dimensional algebra. Then $  D(\Mod A)$ is generically trivial if and only if $A$ is derived equivalent to a hereditary algebra of Dynkin type.
\end{thm}

The result shows that the generically trivial derived categories are rare. Even there is no  infinite continuous family of  indecomposable objects in $K^b(\proj A)$, it is possible that $D(\Mod A)$ has generic objects.

The paper is organized as follows. In section 2 we introduce the background of finite dimensional algebras and triangulated categories. In section 3 we collect some results of endofinite objects in triangulated
category and  give a proof of our main result. In section 4 we give some consequences of the main result and some examples.

\section{Preliminaries}

\subsection{Gentle algebras and repetitive algebras}
Throughout the paper, we fix that the field $k$ is algebraically closed and $A=kQ/I$ is a connected algebra, i.e $Q$ is a connected quiver.
We recall that a quiver is a quadruple $Q=(Q_0,Q_1,s,t)$, where $Q_0$ is the set of
vertices, $Q_1$ is the set of arrows, and $s,t$ are two maps
$Q_1\rightarrow Q_0$ indicating the source and target vertices of an
arrow. A quiver $Q$ is \emph{finite} if $Q_0$ and $Q_1$ are finite sets.

A \emph{relation} $\sg$ on the quiver $Q$ is a $k$ -linear combination of paths $\sg=\sum x_i\a_i$ with paths $\{\a_i\}$ having the same source and target. Let $R_Q$ be the ideal of $kQ$ generated by all
the arrows in $Q$.  A set of relations $\{\rho_i\}$ is called admissible if the ideal $(\rho_i)$ generated by them satisfying that $R_Q^m\subset (\rho_i)\subset R_Q^2$ for some integer $m\geq 2$.
The pair $(Q,{\rho_i})$ is called a quiver with relations $\rho=(\rho_i)$, or a \emph{bounded quiver}. We associate an algebra $k(Q,\rho)=kQ/(\rho)$ to the pair $(Q,\rho)$. The algebra $k(Q,\rho)$ is an
 associative $k$-algebra.

Let $Q$ be a (not necessarily finite) quiver, $\rho$ be a set of relations for $Q$. The bounded quiver $(Q,\rho)$ is called  \emph{special biserial} if it satisfies
\begin{enumerate}
 \item[(1)]  For any vertex in $Q$, there are at most two arrows starting at and ending in this vertex.
\item[(2)] Given an arrow $\b$, there is at most one arrow $\a$ with $t(\a)=s(\b)$ and $\b\a\notin \rho$, and there is at most one arrow $\gamma$ with $t(\b)=s(\gamma)$ and $\gamma\b\notin \rho$.
\item[(3)] Each infinite path in $Q$ contains a subpath which is in $\rho$.
\end{enumerate}

A bounded quiver $(Q,\rho)$ is \emph{gentle} if it is special biserial and the following additional conditions hold,
\begin{enumerate}
 \item[(4)] All elements in $\rho$ are paths of length 2;
\item[(5)]  Given an arrow $\b$, there is at most one arrow $\a'$ with $t(\a')=s(\b)$ and $\b\a\in \rho$, and there is at most one arrow $\gamma'$ with $t(\b)=s(\gamma')$ and $\gamma'\b\in \rho$.
\end{enumerate}

A $k$-algebra $A\cong kQ/(\rho)$ is called \emph{special bieserial} or \emph{gentle} if $(Q,\rho)$ is special biserial or gentle, respectively. A $k$-algebra $A\cong kQ/(\rho)$ is called \emph{string algebra} if
   it is special biserial and the relation is generated by zero relations.

Let $\b\in Q_1$, we denote the formal inverse of $\b$ by $\b^{-1}$ with $s(\b^{-1})=t(\b)$ and $t(\b^{-1})=s(\b)$. A \emph{word} $w$ is a sequence $w_1w_2\ldots w_n$ where $w_i\in Q_1\cup Q_1^{-1}$
and
$s(w_i)=t(w_{i+1})$ for $1\leq i\leq n$. The inverse of a word $w=w_1w_2\ldots w_n$ is $w^{-1}=w_n^{-1}\ldots w_2^{-1}w_1^{-1}$.  A \emph{string} $w$ of length $n>0$ is a word $w$ of length $n$ satisfying that $w_{i+1}\neq w_i^{-1}$ for
$1\leq i\leq n_1$ and $w\notin (\rho)$. For any string $w$, the associated string module $M(w)$ is an indecomposable module \cite{br87}.

Let $A$ be a finite dimensional basic $k$-algebra. Denote by $D=\Hom_k(-,k)$  the
standard duality on $\mod A$. $Q=DA$ is a $A$-$A$-module via $a',a''\in
A, \varphi\in Q,(a'\varphi a'')(a)=\varphi(a'aa'')$.

\begin{defn}The repetitive algebra $\hat{A}$ of $A$ is defined as following,
the underlying vector space is given by
\[\hat{A}=(\oplus_{i\in\mathbb{Z}}A)\oplus(\oplus_{i\in\mathbb{Z}}Q)\]

denote the elements of $\hat{A}$ by $(a_i,\varphi_i)_i$, almost all
$a_i,\varphi_i$ being zero. The multiplication is defined by
\[(a_i,\varphi_i)_i\cdot (b_i,\phi_i)_i=(a_ib_i,a_{i+1}\phi_i+\varphi b_i)_i\]
\end{defn}

An $\hat{A}$-module is given by $M=(M_i,f'_i)_{i\in \mathbb{Z}}$,
where $M_i$ are A-modules and $f'_i:Q\otimes_AM_i\rightarrow M_{i+1}$
such that $f'_{i+1}\circ (1\otimes f'_i)=0$.

Given $\hat{A}$-modules $M=(M_i,f'_i)$ and $N=(N_i,g'_i)$, the
morphism $h:M\rightarrow N$ is a sequence $h=(h_i)_{i\in
\mathbb{Z}}$ such that
\begin{displaymath}
\xymatrix{Q\otimes_A M_i\ar[r]^{f'_i}\ar[d]^{1\otimes h_i}&M_{i+1}\ar[d]^{h_{i+1}}\\
 Q\otimes N_i\ar[r]^{g'_i}&N_{i+1}}
 \end{displaymath}commutes.
There is a dual of the above description \cite{ha88}.

Let $\Mod \hat{A}$ be the category of all $\hat{A}$-modules, and $\mod \hat{A}$ be the subcategory of finite dimensional modules. They both are Frobenius categories \cite{ha88}.
Thus the associated stable categories
$\uMod \hat{A}$ and $\umod \hat{A}$ are triangulated categories. Moreover, $\umod\hat{A}$ is the full subcategory of compact objects in $\uMod\hat{A}$.
Let $  K(\Inj A)$ be the homotopy category of injective $A$-modules. It is  compactly generated triangulated category. Its subcategory $  K(\Inj A)^c$ of all compact objects is triangle equivalent to
the bounded derived category $\db A$ of finitely generated $A$-modules. There are closed relations between $  K(\Inj A)$ and $\uMod \hat{A}$, as well as $\db A$ and $\umod \hat{A}$.
Happel introduced the embedding functor $\db A\rta \umod \hat{A}$ \cite{ha87}, and the functor was extended to $  K(\Inj A)\rta \uMod \hat{A}$ of unbounded complexes in \cite{kl06}.
\begin{rem}
 As in \cite{gk02}, every $\hat{A}$-module $X$ has a maximal injective submodule $X^{\inj}$ since direct limits of injective modules are injective and Zorn's lemma. Let $X^{red}=X/X^{\inj}$, two $\hat{A}$-modules
$X$ and $Y$ are isomorphic as objects in $\uMod \hat{A}$ if and only if $X^{red}$ and $Y^{red}$ are isomorphic in $\Mod \hat{A}$.
\end{rem}

 The embedding functor is very useful because the module category $\mod \hat{A}$ is well-understood, even $\Mod \hat{A}$ is well-understood.
 By the embedding functors, $\db A$ can be controlled by $\umod \hat{A}$ and
$  K(\Inj A)$ can be controlled by $\uMod \hat{A}$.

 The following proposition shows that there exists a fully faithful functor between $  K(\Inj A)$ and $\uMod \hat{A}$ extending Happel's functor.

\begin{prop} \cite[Theorem 7.2]{kl06}\label{inj-ebding}
There is a fully faithful triangle functor $F$ which is the composition of
\[\xymatrix{  K(\Inj A)\ar[r]&  K_{ac}(\Inj\hat{A})\ar[r]^{\sim}&\uMod\hat{A}}\]
and extending Happel's functor
\[\xymatrix{D^b(\mod A)\ar[r]^{-\otimes_AA_{\hat {A}}}&D^b(\mod\hat{A})\ar[r]&\umod\hat{A}}.\]
The functor $F$ admits a right adjoint \[G:\uMod\hat{A}\rta K(\Inj A).\]
\end{prop}

For any algebra $A$, there is some useful results about the image of Happel's functor as following \cite[Lemma 3.4,3.5]{gk02}.
\begin{prop}\label{func-1}
\begin{enumerate}
 \item[(1)]
Let $X\in D^b(\Mod A)$, with $H^i(X)=0$ for $|i|>n$, then $F(X)_i=0$ for $|i|>2(n+1)$.
\item[(2)] Assume $A$ has finite global dimension d. Let $X\in  D^b(\Mod A)$ with $F(X)_i=0$ for $|i|>n$. Then $H^i(X)=0$ for $|i|>(n+1)(d+1)$.
\end{enumerate}
\end{prop}
\subsection{Derived discrete algebras}
For an algebra $A$, let $\db A$ be the bounded derived category of $\mod A$. For a complex $X\in\db A$,
we define the \emph{homological dimension} of $X$ to be the vector
 $h\bdim (X)=(\bdim H^i(X))_{i\in \mabb Z}$.

Algebras with the property that there are only finitely many indecomposable objects in the derived category with the same homological dimension were studied in \cite{v01}. We call them as \emph{derived discrete algebras}, following \cite{ba06}.
The following is a characterization of derived discrete algebras.

\begin{prop}\cite{v01}
Let $A$ be a connected finite dimensional $k$-algebra. Then $A$ is derived discrete if and only if $A$ is either derived hereditary of Dynkin type or is Morita equivalent to the path algebra of a gentle bounded quiver
 with exactly one cycle with different  numbers of clockwise and counterclockwise oriented relations.
\end{prop}

We call a cycle in a quiver satisfies \emph{clock condition} if the numbers of clockwise and counterclockwise oriented relations on the cycle are same.

\section{Generically trivial derived categories}
\subsection{Generic objects of $  D(\Mod A)$}
The endolength of a module was introduced by Crawley-Boevey \cite{cb92}. Modules of finite endolength (endofinite modules) can be used to characterize the representation type of an algebra. We will introduce the generically
trivial triangulated categories and show that endofinite objects play an important role in the structure of triangulated categories.

We mainly concern endofinite objects in the special triangulated categories $  D(\Mod A)$, for some finite dimensional $k$-algebra $A$ .

\begin{defn}
 Let $\mac T$ be a compactly generated triangulated category. An object $N\in \mac T$ is endofinite if the $\End_{\mac T} N$-module $\Hom(C,N)$ has finite length for any $C\in\mac T^c$.
\end{defn}

Endofinite objects of triangulated categories have very nice decomposition properties.
\begin{prop}
 Let $\mac T$ be a compactly generated triangulated category. An endofinite object $X\in \mac T$ has a decomposition $X=\coprod_iX_i$ into indecomposable objects with $\End X_i$ is local, and the decomposition is
unique up to isomorphism.
\end{prop}
 \begin{proof}
See \cite[Proposition 1.2]{kra99}
 \end{proof}
By this proposition, endofinite objects are completely determined by these indecomposable objects. Moreover, the full subcategory of endofinite objects is determined by $\mac T^c$.

By \cite[Theorem 1.2]{kra99}, we know that every endofinite object is pure injective in $\mac T$.
 There is a characterization of endofinite modules, which can be found in \cite{cb92}, originally due to Garavaglia.  An indecomposable module $M$ has finite endolength if and only if every product of copies of $M$ is
isomorphic to a direct sum of copies of $M$.  There is a similar result about endofinite objects in triangulated category \cite{kr01}.
\begin{prop}
 An indecomposable object $X$ is endofinite if and only if every product of copies of $X$ is a coproduct of copies of $X$.
\end{prop}

 Let $A$ be a $k$-algebra. We consider the derived category $  D(\Mod A)$ of unbounded complexes of $A$-modules. It is compactly generated with $  D(\Mod A)^c$ being exactly $K^b(\proj A)$.
For every $X\in   D(\Mod A)$ and $i\in \mabb Z$, the i-th cohomology group $H^i(X)=\Hom(A,X[i])$ has a natural $\End X$-module structure. We have the following characterization for the endofinite objects
in $  D(\Mod A)$ \cite[Lemma 4.1,4.2]{kra02}.

\begin{lem}\label{gen}
\begin{enumerate}
 \item A complex $X\in   D(\Mod A)$ is an endofinite object if and only if  $H^i(X)$ has finite length as an $\End X$-module for every $i\in\mabb Z$.
\item Let $X$ be an endofinite complex in $  D(\Mod A)$. Then $H^i(X)$ is an endofinite $A$-module for all $i\in \mabb Z$.
\end{enumerate}
\end{lem}

A \emph{localizing subcategory} $\mac S$ of $\mac T$ is a triangulated subcategory of $\mac T$ closed under coproducts. If $\mac S$ is generated by compact objects from $\mac T$, then the inclusion $\mac S\rta \mac T$ has
a right adjoint $q:\mac T\rta \mac S$. Moreover, $\mac S$ is a compactly generated triangulated category. There are some relations between endofinite objects in $\mac S$ and endofinite objects in $\mac T$.

\begin{lem}
 Let $\mac S$ be a localizing subcategory of $\mac T$ which is generated by compact objects from $\mac T$, and $q:\mac T\rta\mac S $ be a right adjoint of the inclusion $i:\mac S\rta \mac T$.
 If X is an endofinite object in $\mac T$, then $q(X)$ is endofinite in $\mac S$.
\end{lem}
\begin{proof}
 See \cite[Lemma 1.3]{kra99}
\end{proof}

In general, an endofinite object X in $\mac S$ is not necessary an endofinite object in $\mac T$.
Now we consider the fully faithful functor $F:K(\Inj A)\rta \uMod \hat{A}$  and its right adjoint $G:\uMod \hat{A}\rta K(\Inj A)$ in Proposition \ref{inj-ebding}.
\begin{cor}\label{bk1}
  If $X\in \uMod \hat{A}$ is an endofinite object, then $G(X)\in   K(\Inj A)$ is an endofinite object.
\end{cor}

\begin{defn}
 Let $\mac T$ be a compactly generated triangulated category. An object $E\in \mac T$ is called generic if it is an indecomposable endofinite object and not compact object. $\mac T$ is called generically trivial if it does
not have any generic objects.
\end{defn}

Like generic modules \cite{cb92}, generic objects can be used to characterize the structure of triangulated category.
\begin{exm}
 Let $A$ be a finite dimensional k-algebra and $  D(\Mod A)$ be a compactly generated triangulated category. If there exists a generic module $M\in \Mod A$, then $M$ viewed as a complex concentrated in degree zero, is a
generic object in $  D(\Mod A)$. If $gl.\dim A=\infty$, then any indecomposable object $X$ in $\db A$ not quasi-isomorphic to a perfect complex is a generic object in $  D(\Mod A)$.
\end{exm}

\begin{prop}\label{ger1}
 Let A be a finite dimensional k-algebra with finite global dimension. Then $A$ is derived discrete if and only if  $  D(\Mod A)$ does not contain a generic object $Y$ such that $F(Y)$ is support-finite .
\end{prop}
\begin{proof}
Since $A$ has finite global dimension, there is a triangle equivalence $F: D(\Mod A)\rta \uMod \hat{A}$. The functor $F$ restricted to the subcategory of compact objects is again a triangle equivalence between the full
subcategories of compact objects.

First, suppose that $A$ is not derived discrete, then there is a family of infinitely many indecomposable compact objects $\{X_i\}_{i\in I}\in   D(\Mod A)$ with the same homological dimension
$d=(d_i)_{i\in\mabb Z},d_i\in \mabb N^{A}$.
 By \cite[lemma 3.6]{gk02}, the objects $F(X_i)$ in the family $\{F(X_i)\}_{i\in I}$ are endofinite objects in $\uMod \hat{A}$. Moreover,
the family  $\{F(X_i)\}_{i\in I}$ can be expressed as support-finite $\hat{A}$-modules with the same bounded by Proposition \ref{func-1} .

 Let $\hat{Q}_A$  be the quiver of $\hat{A}$, and $1=e_1+\ldots +e_n$ be the decomposition of identity of primitive idempotent of $A$. Then $\hat{A}$ has primitive idempotents $\mac E_j(e_i)$ for $j\in \mabb Z$. There
are bijection between the vertices of  $\hat{Q}_A$ and  the set $\{\mac E_j(e_i)\}$.
Let $f_j=\sum^n_{i=1}\mac E_j(e_i)$, and $\hat{A}_{m,n}=(\sum^n_{j=m} f_j)\hat{A}(\sum^n_{j=m} f_j)$ be a finite dimensional algebra.
 Thus $ F(X_i)$ could be viewed as $\hat{A}_{m,n}$-module $Y=Y^{red}$ for some $Y\in \Mod \hat{A}_{m,n}$, for some $m,n\in\mabb Z$.

There exists a generic module \cite{cb91} $M$ in $\Mod \hat{A}_{m,n}$, since there are infinitely many indecomposable endofinite $\hat{A}_{m,n}$-modules with same endolength. We can view $M$ as an $\hat{A}$-module
with finite endolength and support-finite.
 By Corollary \ref{bk1}, the complex $G(M)\in   D(\Mod A)$ is an endofinite object, moreover it is a generic object, since the functor $G$ is a triangle equivalence. The object $FG(M)\in \uMod \hat{A}$ is exactly $M$, and is support-finite.

Conversely, assume that there exists a generic object $Y\in  D(\Mod A)$ such that $F(Y)$ is support-finite, we show that $A$ is not derived discrete.
For any compact object $C\in  D(\Mod A)$, we have $\Hom(C,Y)\cong\underline{\Hom}_{\hat{A}}(FC,FY)$ and $\End(Y)\cong \underline{\End}(FY)$.
Thus $F(Y)\in\uMod \hat{A}$ is a generic object. It corresponds a generic module over $\hat{A}_{m,n}$ for some finite dimensional algebra $\hat{A}_{m,n}$. Therefore, there are infinitely many finitely generated
 $\hat{A}_{m,n}$-modules has the same endolength with $F(Y)$ \cite{cb91}. This contradicts to the fact that $\hat{A}$ is representation discrete \cite{v01}.
\end{proof}

\subsection{Generically trivial derived categories}
Assume $A=kQ/I$ is a gentle algebra, where $Q$ has one cycle not satisfying the clock condition \cite{as87} and $I=(\rho)$. Choose a generating set $\rho$ of $I$, by \cite[Proposition 4]{ri97},
the bounded quiver  $(\hat{Q},\mabb Z\rho)$ is
 an expanded gentle quiver. This means that the vertices $a$ are of the following two cases: either $a$ is a crossing vertex: there are exactly two arrows ending in $a$ and two arrows starting in $a$, or else it is a
transition vertex: there just one arrow $\a$ ending in $a$ and just one arrow $\b$, starting in $a$, and $\a\b\notin \rho$. Moreover, if $(Q,\rho)$ is expanded, and $p$ is any path of length at least one in $(Q,\rho)$,
 then there exactly one arrow $\a$ and exactly one arrow $\b$ such that $\a p\b$ is a path in $(Q,\rho)$.

Let  $\hat{A}=k\hat{Q}/\hat{I}$ be the repetitive algebra of $A$. Consider the string algebra $\bar{\hat{A}}=k(\hat{Q},\bar{\hat{\rho}})$, where $\bar{\hat{\rho}}$ is the union of $\hat{\rho}$ and the set of all full paths.
 By the construction of $(\hat{Q},\bar{\hat{\rho}})$, for every transition vertex
there is a unique maximal path starting at and ending in it respectively. For a crossing vertex, there are precisely two maximal paths starting at and ending in it respectively.

Let $A=kQ/I,I=(\rho)$ be a gentle one cycle algebra and derived discrete, and $\hat{A}=k\hat{Q}/\hat{I}$ be the repetitive algebra of $A$.  Denote by $\mu$ the shift map on $(\hat{Q},\hat{\rho})$:
$a[i]\mapsto a[i+1],$ and  $\a[i]\mapsto \a[i+1]$, where $a\in Q_0$ and $\a\in Q_1$. Then $\mu$ is an automorphism of $(\hat{Q},\hat{\rho})$. There is a canonical embedding $\iota:(Q,\rho)\rta (\hat{Q},\hat{\rho})$ given by
$a\mapsto a[0]$ and $\a\mapsto \a[0]$.
Denote by $\mac R$ the finite subset of arrows in $(\hat{Q},\hat{\rho})$ consisting of all arrows $\a[0]$ with $\a\in Q_1$ and the connecting arrows starting at $a[0]$. Then every arrow in $\hat{Q}$ is a $\mu$-shift of an
element in $\mac R$.

The following result is known for experts.  We give an explicit proof for reader's convenience.
\begin{lem}\label{str1}
If $A=kQ/I$ is a gentle one cycle  algebra and derived discrete, then the quiver $(\hat{Q},\hat{\rho})$ of the repetitive algebra $\hat{A}$ has a string of infinite length.
\end{lem}
\begin{proof}
 Let $\hat{A}=k\hat{Q}/\hat{I}$ be the repetitive algebra of $A$. We know that every vertex in $\hat{Q}_0$ is either a crossing or a transition vertex. Given any arrow $\a\in\hat{Q}_1$, we shall find an infinite string starting
 in $s(\a)$.

If $t(\a)$ is a crossing  vertex, there exists a unique minimal direct path $p$ in $\hat{A}$  with $t(p)=t(\a)$ and $s(p)$ crossing.  We denote it by $p_1$.  Consider the vertex $s(p_1)$, there is a unique minimal direct path $p_2\notin (\hat{\rho})$
with $s(p_2)=s(p_1)$. Moreover, $t(p_2)$ is a crossing vertex.
%Consider the vertex $t(p_2)$, we get a unique maximal path $p_3$ with $t(p_3)=t(p_2)$.
 Repeating this procedure, we get a sequence $\ldots p_np_{n-1}\ldots p_2p_1\a$ with $t(p_{2i})=t(p_{2i+1}),s(p_{2i})=s(p_{2i-1})$ for
$i\in\mabb N$ and $t(\a)=t(p_1)$.
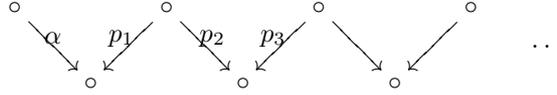
\begin{center}
\begin{tabular}{ccc}
$\xy <1cm,0cm>:
(-1,1)*+{\circ}="1",
(0,0)*+{\circ}="2",
(1,1)*+{\circ}="3",
(2,0)*+{\circ}="4",
(3,1)*+{\circ}="5",
(4,0)*+{\circ}="6",
(5,1)*+{\circ}="7",
(6,0.5)*+{\ldots},
(-0.5,0.6)*+{\alpha},
(0.4,0.6)*+{p_1},
(1.6,0.6)*+{p_2},
(2.4,0.6)*+{p_3},

\ar@{->}"1";"2"<0pt>
\ar@{->}"3";"2"<0pt>
\ar@{->}"3";"4"<0pt>
\ar@{->}"5";"4"<0pt>
\ar@{->}"5";"6"<0pt>
\ar@{->}"7";"6"<0pt>

%\ar@{->}"5,0";"6,-1"<0pt>
%\ar@{->}"5,0";"6,1"<0pt>

\endxy$
\end{tabular}
\end{center}

If $t(\a)$ is a transition vertex, there exists a unique minimal direct path $q$ such that $s(q)=t(\a),p\a\notin (\hat{\rho})$ and $t(q)$ is a crossing vertex.
Then we get a sequence $\ldots p_np_{n-1}\ldots p_2p_1q\a$ with $t(p_{2i})=t(p_{2i+1}),s(p_{2i})=s(p_{2i-1})$ for
$i\in\mabb N$ and $t(q)=t(p_1)$.

There exist at least two arrows $\a_1$ and $\a_2$ in the above sequences with $\a_1=\gamma[j]$ and $\a_2=\gamma[k]$ where $\gamma\in \mac R$ since $\mac R$ is a finite set.
If $(Q,\rho)$ does not satisfy the clock condition, then $i\neq k$ for any pair of arrows $(\a_1,\a_2)$. Otherwise there exists a non-oriented cycle $C= \gamma \ldots\gamma\notin (\hat{\rho})$ \cite[Theorem B]{as87}.
From the finite sequence $\gamma[i]q'\gamma[k]$ and the automorphism $\mu$ of $\hat{Q}$, we get a sequence of infinite length
\[\dots \gamma[2i-k ]q'[i-k]\gamma[i]q'\gamma[k]q'[k-i]\gamma[2k-i]\ldots.\]
This sequence corresponds to a string of infinite length in $(\hat{Q},\hat{\rho})$.
\end{proof}
The existence of infinite strings of $\hat{A}$  implies the existence of generic objects in $\uMod \hat{A}$.
\begin{lem}\label{str2}
 If $A=kQ/I$ is a gentle one cycle  algebra and derived discrete, then there is a generic object in $\uMod\hat{A}$.
\end{lem}

\begin{proof}
 By Lemma \ref{str1}, there exists a string of infinite length in $(\hat{Q},\hat{\rho})$, and the string corresponds to an indecomposable representation of $(\hat{Q},\hat{\rho})$ \cite{kra92}. Moreover, the corresponding string $\hat{A}$-module $M_s$ is locally finite, i.e we have that
$$\dim \Hom_{\hat{A}}(P,M_s)<\infty,$$
 for every indecomposable projective $\hat{A}$-module $P$.
The Hom-space $\Hom(S,M_s)$ is finite dimensional for every simple $\hat{A}$-module $S$.  This implies that for every finite dimensional $\hat{A}$-module $N$, $\Hom(N,M_s)$ is finite
 dimensional. Thus $M_s$ is a generic object in $\uMod \hat{A}$.
\end{proof}

\begin{thm}\label{maint}
  Let $A= kQ/I$ be a finite dimensional $k$-algebra. Then $  D(\Mod A)$ is generically trivial if and only if $A$ is derived equivalent to a hereditary algebra of Dynkin type.
\end{thm}
\begin{proof}
We only need to show the 'only if' part.
Assume there is no generic object in $  D(\Mod A)$, we show that $A$ is derived equivalent to a hereditary algebra of Dynkin type

If $A$ has infinite global dimension, then there exists objects in $\db A$ but not in $  D(\Mod A)^c\cong K^b(\proj A)$. Since $\db A$ is a Hom-finite $k$-linear triangulated category, there exists an indecomposable
object $M\in \db A$ not in $K^b(\proj A)$ which is a generic object in $  D(\Mod A)$.

Assume $gl.\dim A<\infty$. If $A$ is not derived discrete, then there is a triangle equivalence $F:  D(\Mod A)\rta\uMod \hat{A}$. By Proposition \ref{ger1}, there exists a generic object in $\uMod\hat{A}$. Generic objects are preserved under
triangle equivalences. Thus there exists a generic object in $  D(\Mod A)$.
If $A$ is derived discrete and not derived equivalent to an algebra of Dynkin type, then $A$ is a gentle algebra with exactly one cycle in the quiver  $Q$ of $A$ not satisfying the clock condition.
By Lemma \ref{str2}, there exists a generic object in $\uMod \hat{A}$, therefore in $  D(\Mod A)$.
Thus  $A$ is derived equivalent to  a hereditary algebra of Dynkin type.
\end{proof}
\subsection{Applications and examples}
\begin{defn}
If $\mathcal{T}$ is a compactly generated triangulated
category, a triangle  \[\xymatrix{L\ar[r]^f&M\ar[r]^g&N\ar[r]&\Sigma
L}\] in $\mathcal{T}$ is pure-exact if for every object $C$ in
$\mathcal{T}^c$ the induced sequence
\[\xymatrix{0\ar[r]&\mathcal{T}(C,L)\ar[r]&\mathcal{T}(C,M)\ar[r]&\mathcal{T}(C,N)\ar[r]&0}
\] is exact. $\mac T$ is
pure-semisimple if every pure exact triangle splits.
\end{defn}

\begin{prop}\cite[Theorem 12.20]{bel00}\label{ssp1}
 For an Artin algebra $A$, the following are equivalent
\begin{enumerate}
\item $  D(\Mod A)$ is pure semisimple.
\item $A$ is derived equivalent to a hereditary algebra of Dynkin type.
\end{enumerate}

\end{prop}

\begin{cor}
 Let $A$ be a finite dimensional algebra. Then $  D(\Mod A)$ is pure semisimple if and only if $  D(\Mod A)$ is generically trivial.
\end{cor}
\begin{proof}
It  follows easily from Theorem \ref{maint}and Proposition \ref{ssp1}.
\end{proof}
By \cite{xz05}, a $k$-linear triangulated category $\mac C$ is \emph{locally finite} if $\supp \Hom(X,-)$ contains only finitely many indecomposable objects for every indecomposable $X\in\mac C$, where $\supp \Hom(X,-)$ denotes
the subcategory generated by indecomposable objects $Y$ in $\mac C$ with $\Hom(X,Y)\neq 0$.
From the characterizations of locally finite triangulated categories, we have the following result.
\begin{cor}
   Let A be a finite dimensional k-algebra. Then $D(\Mod A)$ is generically trivial if and only if $K^b(\proj A)$ is locally finite.
\end{cor}
\begin{proof}
By Theorem \ref{maint}, we only need to show the 'if' part.
If $K^b(\proj A)$ is locally finite, then the Auslander-Reiten quiver $ \Gamma$ of $K^b(\proj A)$ is isomorphic to $\mabb Z\Delta/G$, where $\Delta$ is a diagram of Dynkin type and $G$ is an automorphism group of $\mabb Z\Delta$ \cite[Theorem 2.3.5]{xz05}. Since $K^b(\proj A)$ has infinitely many indecomposable objects, thus $G$ is trivial. It means that there exists an algebra $B=k\Delta$ such that $K^b(\proj A)\cong K^b(\proj B)$. We have that $A$ and $B$ are derived equivalent. Thus $D(\Mod A)\cong D(\Mod B)$ is generically trivial.
\end{proof}
\begin{exm}\begin{enumerate}
            \item Let $A=kQ/(\rho)$ be the path algebra of the gentle bounded quiver $(Q,\rho)$, which $Q$ is the quiver
$\xymatrix{1\ar@/^/[r]^{\a}&2\ar@/^/[l]^{\b}}$ with relation $\b\a=0$.
It is derived discrete algebra and $gl.\dim A=2$.  The repetitive algebra $\hat{A}$ is the path algebra of locally bounded quiver $(\hat{Q},\hat{\rho})$, where $\hat{Q}$ is
\[\xymatrix{&1[-1]\ar@/^/[d]^{\a[-1]}&&1[0]\ar@/^/[d]^{\a[0]}&&1[1]\ar@/^/[d]^{\a[1]}&&1[2]\ar@/^/[d]^{\a[2]}\\
\ldots\ar[r]&2[-1]\ar@/^/[u]^{\b[-1]}\ar[rr]^{\gamma[-1]}&&2[0]\ar@/^/[u]^{\b[0]}\ar[rr]^{\gamma[0]}&&2[1]\ar@/^/[u]^{\b[1]}\ar[rr]^{\gamma[1]}&&2[2]\ar@/^/[u]^{\b[2]}\ar[r]^{\gamma[2]}&\ldots}\]
with relations $\b[i]\a[i]=0,\gamma[i]\gamma[i-1]=0$ and $\gamma[i]\a[i]\b[i]=\a[i+1]\b[i+1]\gamma[i]$ for all $i\in\mabb Z$.
There is an infinite string of the form
\[\cdots\b^{-1}[i+1]\a^{-1}[i+1]\gamma[i]\b^{-1}[i]\a^{-1}[i]\gamma[i-1]\cdots.\]
The corresponding representation  $M$ of $\hat{Q}$ is
\[\xymatrix{&k\ar@/^/[d]^{\smatrix{1\\0}}&&k\ar@/^/[d]^{\smatrix{1\\0}}&&k\ar@/^/[d]^{\smatrix{1\\0}}&&k\ar@/^/[d]^{\smatrix{1\\0}}\\
\ldots\ar[r]&k^2\ar@/^/[u]^{\smatrix{0&1}}\ar[rr]_{\smatrix{0&0\\1&0}}&&k^2\ar@/^/[u]^{\smatrix{0&1}}\ar[rr]_{\smatrix{0&0\\1&0}}&&k^2\ar@/^/[u]^{\smatrix{0&1}}\ar[rr]_{\smatrix{0&0\\1&0}}&&k^2\ar@/^/[u]^{\smatrix{0&1}}\ar[r]_{\smatrix{0&0\\1&0}}&\ldots.}\]
where $M(\a[i])=\smatrix{1\\0},M(\b[i])=\smatrix{0&1}$ and $M(\gamma[i])=\smatrix{0&0\\1&0}$.

Moreover, $M$ is an endofinite object in $\uMod\hat{A}$ and is indecomposable. Thus it is a generic object in $\uMod \hat{A}$.  By the triangle equivalence $  D(\Mod A)\cong \uMod \hat{A}$, there exists a generic
object in $  D(\Mod A)$.
\item Let $A=kQ/(\rho)$ be the path algebra of the gentle bounded quiver
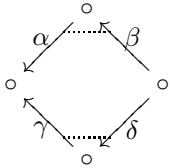
\begin{center}
\begin{tabular}{ccc}
$\xy <1cm,0cm>:
(-1,0)*+{\circ}="1",
(0,1)*+{\circ}="2",
(0,-1)*+{\circ}="3",
(1,0)*+{\circ}="4",
(-0.6,0.6)*+{\a},
(0.6,0.6)*+{\b},
(0.6,-0.6)*+{\delta},
(-0.6,-0.6)*+{\gamma},

\ar@{->}"2";"1"<0pt>
\ar@{->}"3";"1"<0pt>
\ar@{->}"4";"2"<0pt>
\ar@{->}"4";"3"<0pt>
\ar@{..}(0.3,0.7);(-0.3,0.7)<0pt>
\ar@{..}(0.3,-0.7);(-0.3,-0.7)<0pt>
%\ar@{->}"5,0";"6,-1"<0pt>
%\ar@{->}"5,0";"6,1"<0pt>

\endxy$
\end{tabular}
\end{center}
 with $\a\b=0,\delta\gamma=0$. $A$ is tiltable to $A'=k\tilde{A}_3$. Thus we have a triangule equivalence $\Phi:  D(\Mod A')\xrightarrow{\sim}   D(\Mod A)$.
For any generic $A'$-module $M\in\Mod A'$, it is a generic object in $  D(\Mod A')$ by Lemma \ref{gen}. Therefore $\Phi(M)$ is a generic object in $  D(\Mod A)$.
 \end{enumerate}
\end{exm}

\subsection*{Acknowledgements}
This paper is a part of the author's Ph.D thesis, written under supervision of Professor Henning Krause. The author would like to thank him for  suggesting the problem and for many stimulating conversations.
The author is also indebted to Claus M. Ringel for many helpful discussions. Finally, the author would like to express his thanks to Xiaowu Chen and Yong Jiang for their comments on previous versions of this paper.
The author is supported by the China Scholarship Council (CSC).

\bibliographystyle{plain}
%\addcontentsline{toc}{chapter}{Bibliography}
\bibliography{bib/thesis}
\end{document}